\documentclass[12pt]{article}

\usepackage{amsmath, amsthm}
\usepackage{amsfonts}
\usepackage{amsthm}
\usepackage[utf8]{inputenc}

\usepackage{algorithm}

\newtheoremstyle{theorem}%name
  {10pt}		  % space above
  {10pt}  % space below
  {\sl}  % bofy font
  {\parindent}     % ident - empty=no indent,  \parindent= paragraph indent
  {\bf}  % thm head font
  {. }    % punctuation after thm head
  { }    % space after thm head: `` ``=normal \newline=linebreak
  {}     % thm head specification
\theoremstyle{theorem}
\newtheorem{theorem}{Theorem}

\newtheorem{lemma}[theorem]{Lemma}

\newtheoremstyle{defi}%name
  {10pt}		  % space above
  {10pt}  % space below
  {\rm}  % bofy font
  {\parindent}     % ident - empty=no indent,  \parindent= paragraph indent
  {\bf}  % thm head font
  {. }    % punctuation after thm head
  { }    % space after thm head: `` ``=normal \newline=linebreak
  {}     % thm head specification
\theoremstyle{defi}

\newcounter{exampleNo}
\refstepcounter{exampleNo}

\newcommand{\ZZ}{\mathbb{Z}}
\newcommand{\F}{\mathbb{F}}
\newcommand{\KK}{\mathbb{K}}
\newcommand{\NN}{\mathbb{N}}
\newcommand{\bx}{\mathbf{x}}
\newcommand{\lt}{ \mathrm{lt}}

\newcommand{\bu}{\mathbf{u}}
\newcommand{\sbu}{\scriptstyle\mathbf{u}}
\newcommand{\bv}{\mathbf{v}}
\newcommand{\bw}{\mathbf{w}}
\newenvironment{example}[1][Example \arabic{exampleNo}.]{\begin{trivlist}
\item[\hskip \labelsep {\bfseries #1}]\refstepcounter{exampleNo}}{\end{trivlist}}

\begin{document}

\title{A Variant of the Gröbner Basis Algorithm for Computing Hilbert Bases}
\author{\small Natalia Dück$^1$ and Karl-Heinz Zimmermann$^2$\\
\small $^{1,2}$ Hamburg University of Technology\\
%Schwarzenbergstr. 95E, 
\small Hamburg 21073, GERMANY
% First Author's e-mail address\\[2pt]
% Second Author's e-mail address
}
\date{}

\maketitle

\begin{abstract}
Gröbner bases can be used for computing the Hilbert basis of a numerical submonoid.
By using these techniques, we provide an algorithm that calculates a basis of a subspace of a finite-dimensional vector space over a finite prime field given as a matrix kernel.
\end{abstract}

{\bf AMS Subject Classification:} 13P10, 94B05\\
{\bf Key Words and Phrases:} Gröbner basis, integer programming, monoid, Hilbert basis, linear code

\section{Introduction}

Gröbner bases provide a uniform approach to tackling a wide range of problems such as the solvability and solving
algebraic systems of equations, ideal and radial membership decision, and effective computation
in residue class rings modulo polynomial ideals~\cite{adams,becker,cls,sturmfels}.

Furthermore, Gröbner basis techniques are not only a powerful tool for the algorithmic
solution of some fundamental problems in commutative algebra~\cite{buch1}, they also provide means of solving a wide range
of problems in integer programming and invariant theory once these problems have been expressed
in terms of sets of multivariate polynomials~\cite{contitrav-intprogramm, pottier, stu1}.
One such problem is the computation of the Hilbert basis for a submonoid of the numerical monoid $\NN_0^n$.
This problem can be written in terms of polynomials and then be solved using Gröbner basis techniques~\cite{pottier}.
Other elaborations of this method can be found in~\cite{cls-app, stu1}. 

In this paper we will establish an algorithm using Gröbner basis techniques that allows to calculate a basis for a subspace of a finite-dimensional vector space over a finite prime field 
given as a matrix kernel.
This algorithm is based on the one for computing Hilbert bases proposed in~\cite{stu1} and is motivated by the fact that linear codes can be described as such subspaces~\cite{macws, vlint}.

This paper is organized as follows. 
The first section provides an introduction to Gröbner bases, Hilbert bases and their construction for a submonoid of the numerical monoid $\NN_0^n$, and linear codes.
The second section contains the main theorem and a variant of the algorithm for computing a basis for a subspace of $\F_p^n$ described as a matrix kernel, where $p$ is a prime. 
The paper concludes with an example illustrating the algorithm and its application to linear codes.

\section{Preliminaries}
Throughout this paper, $\mathbb{Z}$ denotes the ring of integers, 
$\NN_0$ stands for the set of non-negative integers, 
$\KK$ denotes an arbitrary field, and $\KK[\bx] = \KK[x_1,\ldots,x_n]$ is the commutative polynomial ring in $n$ indeterminates over $\KK$.

\subsection{Gröbner bases}

The \textit{monomials\/} in $\KK[\bx]$ are denoted by $\bx^{\sbu} = x_1^{u_1}x_2^{u_2}\cdots x_n^{u_n}$ and are identified with the lattice 
points $\bu=(u_1,\ldots,u_n)\in\NN_0^n$.
The \textit{degree\/} of a monomial $\bx^{\sbu}$ is the sum $|\bu|=u_1+\cdots+u_n$ and the degree of a polynomial~$f$ is the maximal degree of all
monomials involved in~$f$. 
A \textit{term\/} in $\KK[\bx]$ is a scalar times a monomial.

Denote by $\KK[x_1^{\pm 1},\dots,x_n^{\pm 1}]$
the set of all polynomials given by monomials with exponents in $\mathbb{Z}^n$, which is called the \textit{ring of Laurent polynomials}. 
Negative exponents can be overcome by introducing an additional indeterminate~$t$. 
More precisely, we have
\begin{align}
  \KK[x_1^{\pm 1},\dots,x_n^{\pm 1}]\cong \KK[x_1,\dots,x_n,t]/\left\langle x_1x_2\dots x_nt-1 \right\rangle.\label{eq-laurent}
\end{align}

A \textit{monomial order\/} on $\KK[\bx]$ is a relation $\succ$ on the set of monomials $\bx^{\sbu}$ in $\KK[\bx]$ 
(or equivalently, on the exponent vectors in $\NN_0^n$) satisfying:
(1) $\succ$ is a total ordering, (2) the zero vector $\mathbf{0}$ is the unique minimal element, and
(3) $\bu\succ\bv$ implies $\bu+\bw\succ\bv+\bw$ for all $\bu,\bv,\bw\in\NN_0^n$.
Familiar monomial orders are the lexicographic order, the degree lexicographic order, and the degree reverse lexicographic order.

Given a monomial order $\succ$, each non-zero polynomial $f\in\KK[\bx]$ has a unique \textit{leading term}, denoted by $\lt_\succ(f)$ or simply $\lt(f)$, 
which is given by the largest involved term. 
The coefficient and the monomial of the leading term are called 
the \textit{leading coefficient\/} and the \textit{leading monomial}, respectively.

If $I$ is an ideal in $\KK[\bx]$ and $\succ$ is a monomial order on $\KK[\bx]$, 
its \textit{leading ideal\/} is the monomial ideal generated by the leading monomials of its elements,
\begin{align}
\langle \lt(I)\rangle = \langle\lt(f)\mid f\in I\rangle.
\end{align}
A finite subset $\mathcal{G}$ of an ideal $I$ in $\KK[\bx]$ is a {\em Gröbner basis\/} for $I$ with respect to $\succ$ 
if the leading ideal of $I$ is generated by the set of leading monomials in~$\mathcal{G}$; that is,
\begin{align}
\langle\lt(I)\rangle = \langle\lt(g)\mid g\in \mathcal{G}\rangle.
\end{align}
If no monomial in this generating set is redundant, the Gröbner basis will be called {\em minimal}.
It is called {\em reduced\/} if for any two distinct elements $g,h\in \mathcal{G}$, no term of $h$ is divisible by $\lt(g)$.
A reduced Gröbner basis is uniquely determined provided that the generators are monic.

A Gröbner basis for an ideal $I$ in $\KK[\bx]$ with respect to a monomial order~$\succ$ on $\KK[\bx]$ can be calculated by 
{\em Buchberger's algorithm}.
It starts with an arbitrary generating set for~$I$ and provides in each step new elements of~$I$ yielding eventually a Gröbner basis, which 
can further be transformed into a reduced one. For more about Gröbner basics the reader may consult~\cite{adams, becker, cls}.

\subsection{Monoids, Hilbert bases and their computation using Gröbner bases}

A \textit{monoid\/} is a set $M$ together with a binary operation such that the operation is associative and $M$ possesses an identitiy element.
A \textit{submonoid\/} of a monoid $M$ is a subset of $M$ that is closed under the operation and contains the identity element.
For instance, the set $\NN_0^n$ together with componentwise addition and the zero vector forms a commutative monoid and each submonoid of it is called a {\em numerical monoid}.

%\begin{definition}
A {\em Hilbert basis\/} of a submonoid $K$ of $\NN_0^{n}$ is a minimal (with respect to inclusion) finite subset $\mathcal{H}$ of $K$ such that
each element $k\in K$ can be written as a sum $k=\sum_{h\in\mathcal{H}}c_h h$, where $c_h\in\NN_0$.
%\end{definition}
It is known that each numerical submonoid has a unique Hilbert basis~\cite{sta-enucombi}.

Submonoids arise in various fields like integer programming. 
Such a problem is usually expressed in \textit{standard form}:
\begin{align}
\mbox{Minimize}\quad\mathbf{c}^T\mathbf{x}\quad\mbox{such that}\quad A\mathbf{x}=\mathbf{b}, \; \mathbf{x}\geq 0,
\end{align}
where $\mathbf{b}\in\mathbb{Z}^{m},\mathbf{c}\in\mathbb{Z}^{n}$ and $A\in\mathbb{Z}^{m\times n}$ are given 
and a non-negative integer vector $\mathbf{x}$ is to be found. 
The set of all integer vectors $\mathbf{x}\geq 0$ satisfying the constraint equation $A\mathbf{x}=\mathbf{b}$ is called the \textit{feasible region}.
Of interest here is the case $\mathbf{b}=\mathbf{0}$ because then the feasible region is the kernel of the matrix $A$, written $\ker(A)$, 
which is clearly a numerical submonoid.
The problem is then to find a Hilbert basis of the submonoid $K=\ker(A)$ in $\NN_0^n$, where $A=\left(a_{ij}\right)$ is an $m\times n$ integer matrix.
  
Following \cite{cls-app} we present an algorithm that solves this problem by using Gröbner bases. 
This procedure can also be found in \cite{pottier, stu1}.

The first step is to translate this problem into the realm of polynomials. 
To this end, we associate a variable $x_i$ to every row of $A$, $1\leq i\leq m$. 
Since entries of $A$ can be negative integers, we have to consider the ring of Laurent polynomials. 
Furthermore, define the mapping
\begin{eqnarray}\label{def-psi1}
 \psi:\KK[v_1,\dots,v_n,w_1,\dots,w_n]\rightarrow \KK[x_1^{\pm 1},\dots,x_m^{\pm 1}][w_1,\dots,w_n]
\end{eqnarray}
on the variables first
\begin{eqnarray} \label{def-psi2}
 \psi(v_j)=w_j\prod_{i=1}^{m}x_i^{a_{ij}}\quad\mbox{and}\quad \psi(w_j)=w_j,\quad 1\leq j\leq n,
\end{eqnarray}
and then extend it such that it becomes a ring homomorphism. 
In view of the ideal 
\begin{eqnarray}\label{eq-idealA}
 I_A=\left\langle w_j\prod_{i=1}^{m}x_i^{a_{ij}}-v_j \left|\right. 1\leq j\leq n\right\rangle
\end{eqnarray}
in $\KK[x_1^{\pm 1},\dots,x_m^{\pm 1}][v_1,\dots,v_n,w_1,\dots,w_n]$, we have by~\cite{robbiano} 
\begin{eqnarray}\label{eq-kerA}
\ker(\psi)=I_A\cap\KK[v_1,\dots,v_n,w_1,\dots,w_n]. 
\end{eqnarray}
Using this notation and the polynomial ring in~(\ref{eq-laurent}) instead of the ring of Laurent polynomials, 
we obtain the following assertion due to \cite{stu1}:
\begin{quote}
Let $\mathcal{G}$ be a Gröbner basis for $I_A$ with respect to any monomial order 
for which $x_i\succ v_j$, $t\succ v_j$ and $v_j\succ w_i$ for all $1\leq i\leq m$ and $1\leq j\leq n$.
A Hilbert basis for $K=\ker (A)$ is then given by
\begin{align}
\mathcal{H}=\left\{\alpha\in\NN_0^n\left|\right.\bv^\alpha-\bw^\alpha\in\mathcal{G}\right\}.\label{eg-hilbertA}
\end{align}
\end{quote}
A proof can be found in \cite{stu1}.

This result facilitates an algorithm for computing the Hilbert basis of a given submonoid $\ker(A)$, which is summarized by Algorithm~\ref{alg-hilbertb}.
%For this, note that for calculations the polynomial ring~(\ref{eq-laurent}) is used.
\begin{algorithm}%[!hb]
\caption{Gröbner basis algorithm for computing a Hilbert basis.}
\label{alg-hilbertb}
\begin{enumerate}
\item Associate the ideal $I_A$ defined in~(\ref{eq-idealA}) to a given $m\!\times\!n$ integer matrix~$A$.
\item Compute the reduced Gröbner basis $\mathcal{G}$ for $I_A$ with respect to a monomial order with
	$x_i\succ v_j$, $t\succ v_j$ and $v_j\succ w_k$ for all $1\leq i\leq m$ and $1\leq j,k\leq n$.
\item Read off the elements of the shape $\bv^\alpha-\bw^\alpha$, $\alpha\in\NN_0^n$, which form a Hilbert basis for $\ker(A)$.
\end{enumerate}
\end{algorithm}

\subsection{Linear Codes}
Let $\F$ be the finite field. % with $p$ elements.
A \textit{linear code\/} $\mathcal{C}$ of length $n$ and dimension $k$ over $\F$ is the image of a one-to-one linear mapping 
$\phi:\F^k\rightarrow\F^n$, i.e., $\mathcal{C}=\psi(\F^k)$, where $k\!\leq\!n$.
The code $\mathcal{C}$ is an $[n,k]$ code and its elements are called \textit{codewords}.
In algebraic coding, the codewords are always written as row vectors.

A \textit{generator matrix\/} for an $[n,k]$ code $\mathcal{C}$ is a $k\times n$ matrix $G$ whose rows form a basis of $\mathcal{C}$, 
i.e., $\mathcal{C}=\{\mathbf{a} G\mid \mathbf{a}\in\F^k\}$.
The code $\mathcal{C}$ is in \textit{standard form\/} if it has a generator matrix in reduced echelon form 
$G = \left(I_k\mid M\right)$, where $I_k$ is the $k\times k$ identity matrix.
Each linear code is equivalent (by a monomial transformation) to a linear code in standard form.

For an $[n,k]$ code $\mathcal{C}$ over $\F$, the \textit{dual code} $\mathcal{C}^\perp$ is given by all words $\bu\in\F^n$ such that
$\langle \bu,\mathbf{c}\rangle=0$ for each $\mathbf{c}\in\mathcal{C}$, where $\langle\cdot,\cdot\rangle$ denotes the ordinary inner product.
The dual code $\mathcal{C}^\perp$ is an $[n,n-k]$ code.
If $G = \left(I_k\mid M\right)$ is a generator matrix for $\mathcal{C}$, then $H = \left(-M^T\mid I_{n-k}\right)$ is a generator matrix 
for $\mathcal{C}^\perp$. For each word $\mathbf{c}\in\F^n$, $\mathbf{c}\in \mathcal{C}$ if and only if $\mathbf{c} H^T = \mathbf{0}$.
The matrix $H$ is a \textit{parity check matrix\/} for $\mathcal{C}$~\cite{macws, vlint}.

\section{A Gröbner basis algorithm for finding a Hilbert basis of a matrix kernel}

In the following, let $\F_p$ denote a finite field with $p$ elements, where $p$ is prime.
We are interested in finding the Hilbert basis of the submonoid
\begin{eqnarray}
K=\ker(H_p)\cap\F_p^n, 
\end{eqnarray}
where $H$ is an $m\times n$ integer matrix and $H_p=H\otimes_\ZZ \F_p$.

In other words, we are considering the case in which the numerical monoid $\NN_0^n$ is replaced by the vector space $\F_p^n$ over the finite prime field $\F_p$. 
Then the submonoid $K$ becomes a vector space and the Hilbert basis equals an ordinary basis in the sense of linear algebra.
Clearly, the uniqueness property does no longer hold. 
Nevertheless, the Gröbner basis algorithm for finding a Hilbert basis as described in the previous section (see Algorithm \ref{alg-hilbertb})
can be adapted to this situation in order to find \textit{one} vector space basis. 

Since $p$ is congruent~$0$ in $\F_p$,
the following additional ideal will be used
\begin{align*}
I_p(\mathbf{x})=\left\langle x_i^p-1\left|\right. 1\leq i\leq n\right\rangle. 
\end{align*}
In this way, the exponents of the monomials can be treated as vectors in $\F_p^n$.

Let $H=\left(h_{ij}\right)$ be an $m\times n$-matrix with entries in $\F_p$ and define the ideals
\begin{align}
J_H=\left\langle v_j-w_j\prod_{i=1}^{m}x_i^{h_{ij}}\left|\right. 1\leq j\leq n\right\rangle
\end{align}
and
\begin{align}
I_H=J_H+I_p(\mathbf{x})+I_p(\mathbf{v})+I_p(\mathbf{w}).\label{eq-idealH}
\end{align}

The homomorphism $\psi$ defined in~(\ref{def-psi1}) and~(\ref{def-psi2}) can be used to detect elements in the kernel of $H$.
However, all entries of $H$ can be written (modulo~$p$) as integers  in $\{0,1,\dots,p-1\}$ and so the Laurent polynomials become ordinary polynomials. 
Hence, the image of $\psi$ lies in the polynomial ring $\KK[x_1,\dots,x_m][w_1,\dots,w_n]$.
Note that each non-zero vector $\alpha\in\F_p^n$ can be written as 
\begin{eqnarray}
\alpha=(0,\ldots,0,\alpha_i,\bar{\alpha}), 
\end{eqnarray}
where $\alpha_i\in\F_p\setminus\{0\}$ and $\bar{\alpha}\in\F_p^{n-i}$. 
Furthermore, put
\begin{eqnarray}
\alpha'=\alpha_i\mathbf{e}_i-\alpha= (0,\ldots,0,0,-\bar\alpha), 
\end{eqnarray}
where $\mathbf{e}_i$ is the $i$th unit vector.

\begin{lemma}\label{lem-kernelpsi}
Let $H$ be an $m\times n$-matrix with entries in $\F_p$.
For each non-zero element $\alpha\in\F_p^n$, we have
\begin{align*}
\alpha\in \ker(H)\quad \Longleftrightarrow\quad\psi(v_i^{\alpha_i}\!-\!\bv^{\alpha'}\bw^\alpha)=0~\mod~\left[I_p(\mathbf{x})+I_p(\mathbf{v})\!+\!I_p(\mathbf{w})\right].
\end{align*}

\end{lemma}
\begin{proof} 
All computations will be performed modulo $I_p(\mathbf{x})+I_p(\mathbf{v})+I_p(\mathbf{w})$. 
By the definition of $\psi$, we have 
\begin{align*}
 \psi(v_i^{\alpha_i}-\bv^{\alpha'}\bw^\alpha)&=w_i^{\alpha_i}\prod_{k=1}^{m}x_k^{h_{ki}\alpha_i}
-\bw^\alpha\cdot \bw^{\alpha'}\prod_{i=1}^{n}\prod_{k=1}^{m}x_k^{h_{ki}\alpha_i'}\\
&=w_i^{\alpha_i}\left(\prod_{k=1}^{m}x_k^{h_{ki}\alpha_i}-\prod_{i=1}^{n}\prod_{k=1}^{m}x_k^{h_{ki}\alpha_i'}\right)\\
&=w_i^{\alpha_i}\left(\bx^{H\mathbf{e}_i\alpha_i}-\bx^{H\alpha'}\right).
\end{align*}
In the second equation, 
$\bw^{\alpha'}\bw^{\alpha}=\bw^{\alpha'+\alpha}=\bw^{\mathbf{e}_i\alpha_i}=w_i^{\alpha_i}$. 
Thus 
\begin{align*}
 \psi(v_i^{\alpha_i}\!-\!\bv^{\alpha'}\bw^\alpha)=0\quad&\Longleftrightarrow\quad \bx^{H\mathbf{e}_i\alpha_i}-\bx^{H\alpha'}=0\\
				&\Longleftrightarrow\quad H\mathbf{e}_i\alpha_i-H\alpha'=H\left(\mathbf{e}_i\alpha_i-\alpha'\right)=H\alpha=0\\
				&\Longleftrightarrow\quad \alpha\in\ker(H).
\end{align*}
\end{proof}
Note that $\ker(\psi)$ is a toric ideal~\cite{robbiano}, which can be written as 
\begin{align}
 \ker(\psi)=J_H\cap \KK[\bv,\bw].\label{eq-kerneltoric}
\end{align}

Inspired by the assertion on Hilbert bases for numerical submonoids and based on the previous lemma, we obtain the following main result.
\begin{theorem}\label{thm-main}
Let $\mathcal{G}$ be a Gröbner basis for $I_H$ defined as in (\ref{eq-idealH}) with respect to the lexicographical order with 
$x_1\succ\ldots \succ x_m \succ v_1 \succ \ldots \succ v_n \succ w_1 \succ \ldots\succ w_n$.
Then a basis for $\ker(H)$ in $\F_p^n$ is given by
\begin{align*}
\mathcal{H}=\left\{ (0,\ldots,0,\alpha_i,\bar{\alpha})\in\F_p^n\mid 
v_i^{\alpha_i}-\bv^{\alpha'}\bw^\alpha\in\mathcal{G},\;\alpha'=\alpha_i\mathbf{e}_i-\alpha,\;\alpha_i\ne 0\right\}.\label{eq-hbasis}
\end{align*}
\end{theorem}

Using this assertion, we can obtain an adapted version of Algorithm~\ref{alg-hilbertb} for computing a basis for $\ker(H)$ as a subspace of $\F_p^n$ (see Algorithm~\ref{alg-variant}).
\begin{algorithm}
\caption{Gröbner basis algorithm for computing a basis for $\ker(H)$.}\label{alg-variant}
\begin{enumerate}
\item Associate the ideal $I_H$ defined as in~(\ref{eq-idealH}) to a given  $m\times n$-matrix $H$ over $\F_p$.
\item Compute the reduced Gröbner basis $\mathcal{G}$ for $I_H$ with respect to the lexicographical order with
$x_1\succ\!\ldots\!\succ \!x_m \succ\! v_1 \succ \!\ldots\!\succ \!v_n \succ\! w_1 \succ \!\ldots\!\succ\! w_n.$
\item Read off the elements of the form $v_i^{\alpha_i}-\bv^{\alpha'}\bw^\alpha$ with $\alpha'=\alpha_i\mathbf{e}_i-\alpha$ and $\alpha_i\ne 0$, which give a basis for $\ker(H)$.

\end{enumerate}
\end{algorithm}
For the proof of correctness, which comes hand in hand with the proof of Theorem \ref{thm-main}, three facts will be required:
\begin{enumerate}
\item The reduced Gröbner basis of a binomial ideal consists of binomials~\cite{BStu-Binomial}.
\item The ideal $J_H$ contains no monomials.
\item The ideal $J_H$ is prime and $I_H$ resembles a prime ideal in the following sense: 
If $f,g\in k[\mathbf{x},\mathbf{v},\mathbf{w}]$ are polynomials such that each variable $x_i$ involved in $fg$ has an exponent of at most $p-1$, 
i.e., the exponents of the monomials are written as elements in $\F_p^n$, 
then $fg\in I_H$ implies either $f\in I_H$ or $g\in I_H$.  
 \end{enumerate}

The following proof is an adapted version of the one in~\cite{stu1}. 
Note that all subsequently performed calculations will be either in $\F_p$ or modulo the ideal $I_p(\mathbf{x})+I_p(\mathbf{v})+I_p(\mathbf{w})$.

\begin{proof}
We need to show that the obtained set $\mathcal{H}$ is a minimal spanning set. Assume that this is not the case. Then
there must be a non-zero element $\beta\in\ker(H)$ that cannot be written as a linear combination of elements in $\mathcal{H}$. 
Choose an element $\beta$ such that the 
monomial $\bx^\beta$ is minimal with respect to the chosen monomial order. 
Write $\beta = (0.\ldots,0,\beta_i,\bar\beta)$, where $\beta_i\ne 0$ and $\bar\beta\in\F_p^{n-i}$.
By Lemma~\ref{lem-kernelpsi},~(\ref{eq-kerneltoric}), and $\ker(\psi)\subset J_H$, we obtain
\begin{align*}
f=v_i^{\beta_i}-\bv^{\beta'}\bw^\beta \in J_H.
\end{align*}
Thus $f$ can be reduced to zero on division by $\mathcal{G}$, since $J_H\subset I_H$. 
Hence by the definition of Gr\"obner bases, there must be a polynomial $g\in\mathcal{G}$ with $\lt(g)=v_i^{\gamma_i}$ and $1\leq\gamma_i\leq\beta_i$. 
Put $\delta=\beta_i-\gamma_i$. 
In view of the chosen elimination order and the fact that $\mathcal{G}$ consists of binomials, 
it follows that $g$ is of the form
\begin{align*}
g=v_i^{\gamma_i}-\bv^{\gamma'}\bw^{\eta},
\end{align*}
for some $\gamma'=(0,\ldots,0,-\bar{\gamma})$, where $\bar{\gamma}\in\F_p^{n-i}$, and $\eta\in\F_p^n$.
But by Lemma~\ref{lem-kernelpsi}, the Gr\"obner basis element $g$ will vanish under $\psi$ and so
$$\eta=\gamma_i\mathbf{e}_i+\gamma'=:\gamma.$$
Then we have
\begin{align*}
 f-v_i^{\delta}\cdot g &= v_i^{\beta_i}-\bv^{\beta'}\bw^\beta-v_i^{\delta+\gamma_i}+v_i^\delta \bv^{\gamma'}\bw^{\gamma}\\
		       &= v_i^\delta \bv^{\gamma'}\bw^{\gamma}-\bv^{\beta'}\bw^{\beta}\\
		       &= \bv^{(0\:\dots\:0\:\delta\:-\bar{\gamma})}\bw^{(0\:\dots\:0\:\gamma_i\:\bar{\gamma})}
			  -\bv^{(0\:\dots\:0\:0\:-\bar{\beta})}\bw^{(0\:\dots\:0\:\beta_i\:\bar{\beta})}\\
		       &=\bv^{(0\:\dots\:0\:0\:-\bar{\gamma})}\bw^{(0\:\dots\:0\:\gamma_i\:\bar{\gamma})}
			    \left(v_i^\delta-\bv^{(0\:\dots\:0\:0\:-\bar{\beta}+\bar{\gamma})}\bw^{(0\:\dots\:0\:\delta\:\bar{\beta}-\bar{\gamma})}\right)\\
		       &=\bv^{\gamma'}\bw^{\gamma}\left(v_i^{\delta}-\bv^{-\beta'+\gamma'}\bw^{\beta'-\gamma'+\delta\mathbf{e}_i}\right).
\end{align*}
Applying the previous stated facts~2 and~3 yields 
\begin{align*}
 v_i^{\delta}-\bv^{-\beta'+\gamma'}\bw^{\beta'-\gamma'+\delta\mathbf{e}_i}\in J_H.
\end{align*}
Thus by Lemma~\ref{lem-kernelpsi}, $\beta'-\gamma'+\delta\mathbf{e}_i\in\ker(H)$. 
But by the choice of $g$, $\delta<\beta_i$ and so $\bx^{\beta'-\gamma'+\delta\mathbf{e}_i}\prec \bx^{\beta}$.
Hence by the selection of $\beta$, $\beta'-\gamma'+\delta\mathbf{e}_i$ can be written as a linear combination of elements in $\mathcal{H}$. 
The same holds for $\gamma$, since it lies in $\mathcal{H}$ due to the choice of the corresponding Gröbner basis element $g$. 
But then
\begin{align*}
 \beta=\beta'+\beta_i\mathbf{e}_i=\beta'+\left(\delta+\gamma_i\right)\mathbf{e}_i+\gamma'-\gamma'
=\left(\beta'-\gamma'+\delta\mathbf{e}_i\right)+\gamma,
\end{align*}
and so $\beta$ can also be written as such a linear combination contradicting the choice of $\beta$ and hence proving the assertion. 
\end{proof}

We conclude by giving an example illustrating applications to linear codes.

\begin{example}
 Consider the $[11,6]$ ternary Golay code~\cite{macws, vlint} with the generator matrix $G=\left(I_6\left|\right.M\right)$, where
\begin{align*}
 M=\begin{pmatrix}
    1&1&1&1&1\\
    0&1&2&2&1\\
    1&0&1&2&2\\
    2&1&0&1&2\\
    2&2&1&0&1\\
    1&2&2&1&0
   \end{pmatrix}.
\end{align*}
Then a parity check matrix is 
\begin{align*}
 H=\left(\begin{array}{ccccccccccc}
    2&0&2&1&1&2&1&0&0&0&0\\
    2&2&0&2&1&1&0&1&0&0&0\\
    2&1&2&0&2&1&0&0&1&0&0\\
    2&1&1&2&0&2&0&0&0&1&0\\
    2&2&1&1&2&0&0&0&0&0&1
   \end{array}\right).
\end{align*}
Applying Algorithm \ref{alg-variant} for computing a basis of $\ker(H)$ yields the following polynomials belonging to the reduced Gr\"obner basis
\begin{align*}
 &v_6-v_7^2v_8v_9v_{10}^2w_6w_7w_8^2w_9^2w_{10},\\
 &v_5-v_7v_8v_9^2v_{11}^2w_5w_7^2w_8^2w_9w_{11},\\
 &v_4-v_7v_8^2v_{10}^2v_{11}w_4w_7^2w_8w_{10}w_{11}^2,\\
 &v_3-v_7^2v_9^2v_{10}v_{11}w_3w_7w_9w_{10}^2w_{11}^2,\\
 &v_2-v_8^2v_9v_{10}v_{11}^2w_2w_8w_9^2w_{10}^2w_{11},\\
 &v_1-v_7^2v_8^2v_9^2v_{10}^2v_{11}^2w_1w_7w_8w_9w_{10}w_{11}.
\end{align*}
The Hilbert basis taken from these polynomials corresponds to the row vectors of the matrix $G$.
\end{example}

\bibliographystyle{plain}
\bibliography{refBook,refPaper}

\end{document}